\documentclass[11pt, twoside, letter]{amsart}
\usepackage[]{appendix}
\usepackage{matthew}
\usepackage{matthew-cd}
\usepackage{hhline}
\usepackage{mathrsfs}
\usepackage{adjustbox}

\usepackage{algorithm}
\usepackage{algorithmicx}
\usepackage{algpseudocode}

\usepackage[pdftex,plainpages=false,hypertexnames=false,pdfpagelabels]{hyperref}
\newcommand{\arXiv}[1]{\href{http://arxiv.org/abs/#1}{\tt arXiv:\nolinkurl{#1}}}
\newcommand{\arxiv}[1]{\href{http://arxiv.org/abs/#1}{\tt arXiv:\nolinkurl{#1}}}
\newcommand{\doi}[1]{\href{http://dx.doi.org/#1}{{\tt DOI:#1}}}

\newcommand{\googlebooks}[1]{(preview at \href{http://books.google.com/books?id=#1}{google books})}
\renewcommand{\MR}[1]{}

\usepackage{xcolor}
\definecolor{dark-red}{rgb}{0.7,0.25,0.25}
\definecolor{dark-blue}{rgb}{0.15,0.15,0.55}
\definecolor{medium-blue}{rgb}{0,0,.8}
\definecolor{DarkGreen}{RGB}{0,150,0}
\hypersetup{
   colorlinks, linkcolor={purple},
   citecolor={medium-blue}, urlcolor={medium-blue}
}

\usepackage{etoolbox}
\makeatletter
\patchcmd{\@maketitle}
  {\ifx\@empty\@dedicatory}
  {\ifx\@empty\@date \else {\vskip3ex \centering\footnotesize\@date\par\vskip1ex}\fi
   \ifx\@empty\@dedicatory}
  {}{}
\patchcmd{\@adminfootnotes}
  {\ifx\@empty\@date\else \@footnotetext{\@setdate}\fi}
  {}{}{}
\makeatother

\numberwithin{table}{section}

\title{Geometric Invariants for Fusion Categories}
\author{Tobias Hagge \and Matthew Titsworth}
\email{tobiashagge@gmail.com, matthew.titsworth@utdallas.edu}
\date{\today}
\begin{document}
\maketitle

\begin{abstract}
The problem of determining gauge and monoidal equivalence classes of fusion categories is considered from the perspective of geometric invariant theory.
It is shown that the gauge (or monoidal) class of a fusion category is determined by the evaluation of a finite set of gauge (or monoidal) invariant functions.
In the multiplicity free case this leads to a fast algorithm for computing a classifying set of functions.
\end{abstract}

\section{Introduction}
A fundamental problem in the study of fusion categories\cite{MR2183279} is to classify all fusion categories over a field $k$ with Grothendieck ring isomorphic to a based ring $(R,B)$ up to monoidal equivalence.
This is extremely difficult in general; a classification of all fusion categories would necessarily include a classification of all finite groups.
It is, however, sufficient to solve a set of polynomial equations known as the {\em pentagon equations}. 

There are a number of practical difficulties to this approach.
First, the number of equations and variables involved tends to grow exponentially as the number of simple objects increases.
Gr\"obner basis methods for solving such systems have worst-case complexity which is double-exponential in the size of the input.
Second, although Ocneanu rigidity (see \cite{MR2183279}) guarantees that there are only a finite number of equivalence classes of fusion categories with a given Grothendieck ring, solutions are only unique up to gauge transformations given by a group $G$ and automorphisms of the Grothendieck ring given by a group $Aut(R,B)$.
Thus it is often necessary to check solutions for equivalence.

In the presence of extra structure, there are a number of techniques to make classification easier.
For example, pointed categories with Grothendieck ring determined by a finite group $G$ are classified by $H^{3}(G,k^\times)$.
Every unitary fusion category comes from a finite index subfactor\cite{MR1055708}. 
Certain fusion categories can be realized using Cuntz algebras\cite{MR3167494} or Leavitt algebras\cite{2015arXiv150603546E}. 
Monoidal algebras\cite{MR1237835,MR2132671} have been used to classify all fusion categories coming from $A$-series Lie algebras and braided fusion categories coming from $B$, $C$, and $D$-series Lie algebras.

Fusion categories are often classified using invariants.
In \cite{2014arXiv1405.7950B} it is shown that two Tambara-Yamagami categories\cite{MR1659954} $\mathcal C$ and $\mathcal C'$ are equivalent if all of their higher Frobenius-Schur indicators as introduced by \cite{MR2381536} are equal. 
Categorical invariants such as the Turaev-Viro\cite{MR1191386} and Reshetikin-Turaev\cite{MR1036112,MR1091619} invariants become invariants of (subclasses of) fusion categories when the manifold is fixed and the category is taken as a parameter.
It is conjectured that the $(S,T)$ matrices classify all modular categories \cite{2013arXiv1310.7050B}.

In this work, we develop a new class of invariants strong enough to classify arbitrary fusion categories.
These invariants are computable and for a large class of fusion categories we show that our method leads to fast algorithms for computing gauge classes, monoidal classes and the Galois action on these classes.
For a subset of these, we show how to determine values of invariants without first solving the pentagon equations.

Our approach uses the machinery of geometric invariant theory (see \cite{MR1304906} or \cite{MR2004511} for an introduction).
Specifically, given a based ring $(R,B)$, the pentagon equations generate an ideal in some polynomial $k$-algebra $k[\Phi(R,B)]$. 
From this, one can construct an algebraic scheme \footnote{A separated affine scheme of finite type over $spec(k)$. These will be defined from affine algebraic sets, and by abuse of notation will also be denoted $X(R,B)$. These sets will themselves be useful, but the structure sheaf will carry all of the important information.} $X(R,B)$ (see Definition \ref{def:fusionsystem} and Remark \ref{rem:scheme}) which is acted upon by gauge group $G$ (Definition \ref{def:gaugetransform}) and Grothendieck ring automorphism group $Aut(R,B)$ (Definition \ref{def:autaction}).
Theorems of \cite{MR1304906} then reduce the problem of determining gauge and monoidal classes to one of constructing regular functions on $X(R,B)$ invariant under the actions of $G$ and $Aut(R,B)$.
In particular, for a non-empty scheme $X(R,B)$ associated to $(R,B)$ we prove the following:
\begin{theorem}\label{thm:main1}
Fix a based ring $(R,B)$.
Define $\mathscr N$ to be the number of gauge equivalence classes for fusion categories Grothendieck equivalent to $(R,B)$ and $\mathscr M$ the number of monoidal equivalence classes. 
For all $\mathbf F \in X(R,B)$, denote by $\mathcal C(\mathbf F)$ its associated fusion category.
Then there exists $\mathscr N-1$ $G$-invariant regular functions $f_{1},\ldots,f_{\mathscr N-1}$ on $X(R,B)$ such that for all $\mathbf F,\mathbf F' \in X(R,B)$, $\mathcal C(\mathbf F)$ and $\mathcal C(\mathbf F')$ are gauge equivalent if and only if
\begin{align}
f_i(\mathbf F')=f_i(\mathbf F), i=1,\ldots, {\mathscr N-1}
\end{align} 
Similarly, there exist $\mathscr M-1$ $Aut(R,B) \ltimes G$-invariant regular functions $g_1,\ldots, g_{\mathscr M-1}$ on $X(R,B)$ such that for all $\mathbf F, \mathbf F' \in X(R,B)$, $\mathcal C(\mathbf F)$ and $\mathbf F'$ are monoidally equivalent if and only if
\begin{align}
g_j(\mathbf F')=g_j(\mathbf F), j=1,\ldots, {\mathscr M-1}
\end{align} 
\end{theorem}
For the case of multiplicity free fusion categories (i.e. all structure coefficients for the Grothendieck ring are $0$ or $1$) we obtain a stronger statement:
\begin{theorem}\label{thm:main2}
Fix a based ring $(R,B)$.
Define $\mathscr N$ to be the number of gauge equivalence classes for fusion categories Grothendieck equivalent to $(R,B)$ and $\mathscr M$ the number of monoidal equivalence classes. 
Then there exist $$P\leq{\mathscr N \choose 2}$$ $G$-invariant rational monomials 
\begin{align}\label{eq:invmons}
m_i&=(\phi_{i,1})^{k_{i,1}}\ldots(\phi_{i,j})^{k_{i,j}}, &\text{ with }& &\phi&\in \Phi(R,B) &\text{ and }& &k&\in\Z
\end{align}
such that for all $\mathbf F,\mathbf F' \in X(R,B)$, $\mathcal C(\mathbf F)$ and $\mathcal C(\mathbf F')$ are gauge equivalent if and only if
\begin{align}
F(m_i)&=F'(m_i),& i=1,\ldots P
\end{align}
or both are simultaneously undefined.
Similarly, there exist $$Q \leq {\mathscr M \choose 2}$$ $Aut(R,B)$-invariant linear combinations $l_1,\ldots, l_Q$ of $G$-invariant rational monomials as in \eq{invmons} such that $\mathcal C(\mathbf F)$ and $\mathcal C(\mathbf F')$ are monoidally equivalent if and only if
\begin{align}
F(l_j)&=F'(l_j),& j=1,\ldots Q.
\end{align}
\end{theorem}

Computing these quantities involves a technical difficulty.
$G$-invariant rational monomials form a (finitely generated free Abelian) group.
Naively, one would like to compute a basis for this group and evaluate $\mathbf F\in X(R,B)$ on the basis monomials.
Unfortunately, if $0$ is in the image of $F$ the evaluation map on monomials is not totally defined and does not satisfy the homomorphism property.
However, the sets $F^{-1}(0)$ are gauge invariant and given $F^{-1}(0)$ there exists a fast algorithm for computing a set of nonzero invariant monomials which uniquely determine the evaluations on all other monomials\footnote{A priori, only the non-zero evaluations are determined, but as shown in Appendix \ref{app:important} all other evaluations are determined by the non-zero ones.}.

\begin{theorem}\label{thm:main2a}
Let $(R,B)$ be multiplicity free and $F^{-1}(0)$ be the pre-image of $0$ for some $\mathbf F \in X(R,B)$.
Then for all $\mathbf F'\in X(R,B)$ with $(F')^{-1}(0)=F^{-1}(0)$, the evaluations of $\mathbf F'$ on all $G$-invariant rational monomials are completely determined by the non-zero evaluations of $\mathbf F'$ on a finite set $\mathcal B$.
Given $(R,B)$ and $F^{-1}(0)$ a choice of $\mathcal B$ is computable with worst case runtime complexity\cite{MR2049753} $$\mathcal O(\vert \Phi\vert\vert \Gamma \vert^4 \log^2(\mathcal M)).$$ 
\end{theorem}

Together with Theorem \ref{thm:main2} one obtains

\begin{corollary}\label{cor:compute}
Let $\mathbf F, \mathbf F'\in X(R,B)$ and $\rho \in Aut(R,B)$ such that $(\rho \cdot (F')^{-1}(0))=F^{-1}(0)$. 
$\mathcal C(\mathbf F)$ and $\mathcal C(\mathbf F')$ are then monoidally equivalent if and only if 
\begin{align}
F(m)&=F'(\rho\cdot m) & \forall m \in \mathcal B.
\end{align}

\end{corollary}

This falls short of an algorithmic proof of Theorem 1.2, but for practical purposes is just as good.
All of Theorem \ref{thm:main1} through Corollary \ref{cor:compute} are consequences of the following: 
\begin{theorem}\label{thm:main3}
Let $X(R,B)$ be a fusion scheme acted upon by a gauge group $G$ and automorphism group $Aut(R,B)$.
Then there exists an algebraic scheme $Y(R,B)$ and surjective morphism of schemes $$\psi:X(R,B)\rightarrow Y(R,B)$$ such that $(Y(R,B),\psi)$ is a geometric quotient of $X(R,B)$ by $G$ in the sense of \cite[Definition 0.6]{MR1304906}.
Similarly, there exists an algebraic scheme $Z(R,B)$ and surjective morphism of schemes $$\eta:Y(R,B)\rightarrow Z(R,B)$$ such that $(Z(R,B),\psi\circ \eta)$ is a geometric quotient of $X(R,B)$ by $Aut\ltimes G$. 
\end{theorem}

The structure of this paper is as follows:
Section \ref{s:background} defines fusion categories and the algebraic sets associated to them.
In section \ref{s:quotients}, the algebraic groups $G$ and $Aut(R,B) \ltimes G$ are defined, fusion schemes are defined, and Theorems \ref{thm:main1}, \ref{thm:main3}, and \ref{thm:main2} are proven.
Section \ref{s:practical} addresses computation issues, proves Theorem \ref{thm:main2a}, and an algorithm is given for computing a classifying set of invariants. 
Finally, Section \ref{s:examples} provides concrete examples.
The quotient scheme associated to the Fibonacci and Yang-Lee categories is constructed in Section \ref{ss:exfib}.
As a more significant example, in section \ref{ss:ds3} solutions the pentagon equations for the Grothendieck ring of $Rep(D(S_3))$ are used to compute monoidal classes and show the existence of a new fusion category which does not admit braiding.

\begin{remark}
While pivotal structures and braidings are ignored in the current work, the results presented carry over; the braiding and pivotal constants are also acted upon algebraically by $G$ and there are in fact invariant functions which distinguish braided and pivotal equivalences classes of fusion categories.
The argument follows lines similar to those given in this document and is omitted for brevity.
Still, in presenting monoidal classes for $Rep(D(S_3))$ it is worth while to note which modular structures are associated to the underlying fusion categories, so this is done.
\end{remark}

\begin{remark}
The solutions to pentagon equations for $Rep(D(S_3))$ are rather large and unwieldy, so we have opted to instead include them in a Mathematica package designed to manipulate the arithmetic data for fusion categories.
This package can be found at \href{http://mktitsworth.com/fusion-categories/}{http://mktitsworth.com/fusion-categories} and includes example data for many other categories as well.
\end{remark}

\section{Background}\label{s:background}
In the sequel, we assume $k$ is an algebraically closed field of characteristic $0$.
We first recall the definitions of based rings and fusion categories.

\begin{definition}[{\cite[Definitions 2.1 and 2.2]{MR1976459}}]\label{def:basedring}
A {\em unital based ring} $(R,B)$ is a $\Z^+$- ring $R$ together with a set $B\subset R$ and identity $1_R\in B$ such that
\begin{itemize}
	\item (Structure constants) There exist non-negative integers $N_{XY}^Z$ for $X,Y,Z\in B$ such that 
	\begin{displaymath}
	X Y = \sum_{Z\in B}N_{XY}^Z.
	\end{displaymath}
	\item (Duality) A bijection $*:B\rightarrow B$ such that $1_R^*=1_R$ which extends to an anti-involution on $(R,B)$, i.e. $(XY)^*=X^*Y^*,\forall X,Y\in B$.
\end{itemize}

By associativity the structure constants must satisfy 
\begin{equation}
\sum_{U\in B}N_{XY}^U N_{UZ}^W = \sum_{V\in B}N_{YZ}^V N_{XV}^W
\end{equation}
for all $U,V,W,X,Y,Z\in B$.

The structure constants define a map $N:B^{\times 3}\rightarrow \Z^+$ which we can extend recursively to arbitrary $n+1>3$ via
$$
N_{X_1 \ldots X_n}^{Y}:=\sum_{Z\in B}N_{X_1\ldots X_{n-1}}^{Z}N_{Z X_n}^Y
$$
for all $X_1,\ldots, X_n, Y\in B$.
Those structure constants which are non-zero will play an important role and so we define
$$\Gamma(R,B)=\{(X,Y,Z)\in B^{\times 3}\vert N_{XY}^Z\neq 0\}$$
with $\gamma_{XY}^Z$ the notation for $(X,Y,Z)$. 
We will say that $(R,B)$ is {\em multiplicity free} if $N(\Gamma(R,B))=\{1\}$.

For based rings $(R,B)$ and $(R,B')$, every bijection $\rho:B\rightarrow B'$ satisfying
\begin{equation}
N_{XY}^Z=N_{\phi(X)\phi(Y)}^{\phi(Z)},\forall X,Y,Z\in B
\end{equation}
defines a unique based ring isomorphism $\rho':(R,B)\rightarrow (R,B')$ and vice versa.
Let $Aut(R,B)$ to be the group of based ring automorphisms of $(R,B)$.
\end{definition}

\begin{definition}\label{def:fusioncategory}
A {\em fusion category} $\mathcal C$ over $k$ is a monoidal semi-simple Abelian category\footnote{We will denote the monoidal bifunctor of a fusion category $\mathcal C$ by $\otimes$ and direct sum of objects by $\oplus$} with identity $\mathbf 1$ such that
\begin{enumerate}
	\item ($k$-linearity) $\mathcal C$ is enriched over $Vec_{Fin}(k)$. This is to say that $\mathcal C(a,b)$ is a finite dimensional vector space over $k$ for all objects $a,b\in \mathcal C_0$.
	\item (Finiteness) There are finitely many isomorphism classes of simple objects in $\mathcal C_0$ and $\mathcal C(a,a)\cong k$ for all simple objects $a\in \mathcal C_0$.
	\item (Rigidity) For every object $a\in \mathcal C_0$, there is an object $a^*\in \mathcal C_0$ and evaluation and co-evaluation maps
	\begin{align*}
	ev_{a}:a \otimes a^*&\rightarrow \mathbf 1 & coev_{a}:\mathbf 1 &\rightarrow a^*\otimes a
	\end{align*}
	such that
	\begin{align}
	\lambda_a \circ (ev_a\otimes Id_a)\circ \alpha_{a,a^*,a}\circ (Id_a \otimes coev_a)\circ \rho_a^{-1}&=Id_a \text{ and }\\
	\rho_{a^*} \circ (Id_{a^*}\otimes ev_{a^*})\circ \alpha_{a^*,a,a^*}^{-1} \circ (coev_{a^*}\otimes Id_{a^*}) \circ \lambda_{a^*}^{-1} &=Id_{a^*}.
	\end{align}
\end{enumerate}

We denote by $\mathcal K_0(\mathcal C)$ the Grothendieck ring of $\mathcal C$.
This is a based ring as defined in \ref{def:basedring} with basis elements corresponding to equivalence classes of simple objects $a,b,\ldots \in \mathcal C_0$ and multiplication induced from $\otimes$ via $N_{a_1\ldots a_n}^b=dim(\mathcal C(a_1\otimes\ldots \otimes a_n, b))$.
We say that $\mathcal C$ is multiplicity free if $\mathcal K_0(\mathcal C)$ is multiplicity free.
Given two fusion categories $\mathcal C$ and $\mathcal D$ we say that they are {\em Grothendieck equivalent} if and only if $\mathcal K_0(\mathcal C)\cong \mathcal K_0(\mathcal D)$.
\end{definition}

We say that a based ring $(R,B)$ {\em admits categorification} if there exists a fusion category $\mathcal C$ such that $(R,B)\cong\mathcal K_0(\mathcal C)$.
The questions are then: (1) given a based ring $(R,B)$ does it admit a categorification and (2) if so, how many?
It is well known that $(R,B)$ admits categorification if and only if there exists a solution to the pentagon equations (see any of \cite{MR1659954,MR1797619,2013arXiv1305.2229D} for the standard construction).
The solutions to these equations define an affine algebraic set $X$ in the classical sense, i.e. they define a set of points in some affine $k$-space.

\begin{definition}\label{def:fusionsystem}
Fix a based ring $(R,B)$.
For every $(a,b,c,d)\in B^{\times 4}$, such that $N_{abc}^d\neq 0$ define the matrix $\Phi_{abc}^d\in Mat_{N_{abc}^d\times N_{abc}^d}(k)$ with entries denoted
\begin{displaymath}
\tenj{\Phi_{abc}^d,i,e,j,i',f,j'},\left.\begin{array}{l} f\in B \\ e \in B\end{array}\right.,\left.\begin{array}{l} i'\in \{1,\ldots, N_{af}^d\} \\ i \in \{1,\ldots, N_{ab}^e\} \end{array}\right.,\left.\begin{array}{l} j'\in \{1,\ldots, N_{bc}^f\} \\ j \in \{1,\ldots, N_{ec}^d\} \end{array}\right.
\end{displaymath}
and define $\Phi(R,B)$ to be the set of entries in all $\Phi_{abc}^d$.
The ideal in $k[\Phi(R,B)]$ generated by 
\begin{align}
	&\Phi_{a \mathbf 1 b}^c=Id_{N_{ab}^c} \label{eq:F1} \\
	&\tenj{\Phi_{a a^* a}^a,1,1,1,1,1,1}\neq 0 \label{eq:F2} \\
	&\sum_{n''=1}^{N_{fh}^e}\tenj{\Phi_{fcd}^e,m,i,n,l'',h,n''}\tenj{\Phi_{abh}^e,l,f,n'',m''',j,n'} = \label{eq:F3} \\
	&\sum_{g\in L}\sum_{l'=1}^{N_{bc}^g}\sum_{m'=1}^{N_{ag}^i}\sum_{m''=1}^{N_{gd}^j}\tenj{\Phi_{abc}^d,l,f,m,l',g,m'}\tenj{\Phi_{agd}^e,m',i,n',m'',j,n''}\tenj{\Phi_{bcd}^j,l',g,m'',l'',h,m'''} \nonumber
\end{align}
together with the condition that the $\Phi_{abc}^d$ be invertible defines an algebraic set $X(R,B)$.
A point $\mathbf F\in X(R,B)$ is a map $F:\Phi(R,B)\rightarrow k$ such that the equations defined by \eq{F1}-\eq{F3} are satisfied.
We will also fix the notation $\tenj{F_{abc}^d,i,e,j,i',f,j'}:=F(\tenj{\Phi_{abc}^d,i,e,j,i',f,j'})$ and $F_{abc}^d$ for the matrix with entries $\tenj{F_{abc}^d,i,e,j,i',f,j'}$.
\end{definition}

Given a fusion category $\mathcal C$ with $\mathcal K_0(\mathcal C)\cong (R,B)$, one can define a point $\mathbf F_\mathcal C\in X(R,B)$.
Given a point $\mathbf F\in X(R,B)$, one can construct a fusion category $\mathcal C(\mathbf F)$ such that $\mathcal K_0(\mathcal C)\cong (R,B)$.
Starting with $\mathcal C$, one has that $\mathcal C(\mathbf F_\mathcal C)\cong \mathcal C$\cite[Lemma 3.4 and Proposition 3.7]{2013arXiv1305.2229D}.

\section{Invariants of fusion schemes}\label{s:quotients}
For brevity, fix a based ring $(R,B)$ and let $X=X(R,B)$, $\Gamma=\Gamma(R,B)$, and $\Phi=\Phi(R,B)$.

\subsection{Group actions and quotients}
\begin{definition}\label{def:gaugetransform}
Given $X$, let $G=\prod_{\gamma_{ab}^c\in \Gamma}GL_{N_{ab}^c}(k)$.
For $g\in G$, fix the notations $$g_{ab}^c\in GL_{N_{ab}^c}(k) \text{ and } g^{ab}_c=(g_{ab}^c)^{-1}.$$
Then $$g=\prod_{\gamma_{ab}^c\in \Gamma}(g_{ab}^c)\text{ and }g^{-1}=\prod_{\gamma_{ab}^c\in \Gamma}(g^{ab}_c).$$

Let $G$ act algebraically on the point $\mathbf F \in X$ such that for $g\cdot \mathbf F$, the induced map $(g\cdot F)$ is given by
\begin{equation}\label{eq:gaugemult}
\tenj{(g\cdot F)_{abc}^{d},i',e,j',k',f,l'}:=\sum_{i=1}^{N_{ab}^e}\sum_{j=1}^{N_{ec}^d}\sum_{k=1}^{N_{bc}^f}\sum_{l=1}^{N_{af}^d}(g_{ab}^e)_i^{i'}(g_{ec}^d)_j^{j'}\tenj{F_{abc}^d,i,e,j,k,f,l}(g_d^{af})_{l'}^l(g_f^{bc})_{k'}^k.
\end{equation}
It is straight forward to check that $g\cdot \mathbf F$ is a point of $X$.
Call $\mathbf F,\mathbf F'\in X$ {\em gauge equivalent} and write $\mathbf F \sim \mathbf F'$ if they belong to the same orbit $G\cdot\mathbf F$ of the $G$-action.
Say that $\mathcal C(\mathbf F)$ and $\mathcal C(\mathbf F')$ are gauge equivalent if and only if $\mathbf F \sim \mathbf F'$.
\end{definition}

\begin{remark}\label{rem:scheme}
For any algebrac set $X$ associated to the zero set of a polynomial $I$, the {\em ring of regular functions} (with respect to $I$) on $X$ is $$O_X(X):= k[x_1, ..., x_n]/I.$$ 
There is a natural surjection from $O_X(X)$ to restrictions of polynomial functions on $k^n$ to $X$ and this is a bijection if $I$ is a radical ideal. 
Given any two points in an algebraic set $X$, there is a polynomial (and thus a regular function) which distinguishes them.

One wishes to imitate this behavior at the level of orbits, that is, to find polynomial functions which distinguish the orbits of $X$ under the action of the gauge group $G$ and later the permutation group $Aut(R,B)$ as well. 
Let $O_X(X)^G$ be the ring of $G$-invariant regular functions on $X$, that is equivalence classes of $G$-invariant polynomial functions defined on $X$. 
The desired polynomial functions must then be equivalence class representatives of elements of $O_X(X)^G$.

Mumford's geometric invariant theory\cite{MR1304906} establishes that under suitable conditions, there is an object $Y$, the ``points'' of which correspond to orbits in $X$ and the ``regular functions'' of which form the ring $O_X(X)^G$. 
Furthermore, given any two orbits there is an element of $O_X(X)^G$ which distinguishes them.

The object $Y$ is not an algebraic set, but rather a scheme, or in other words, a locally ringed space isomorphic to the spectrum of a commutative ring.
Given an ideal $I$ and corresponding algebraic set $X$, Mumford gives $X$ a scheme structure by constructing its structure sheaf $\mathcal O_X$; the pair $(X,\mathcal O_X)$, and by abuse of notation $X$, shall be referred to as a {\em fusion scheme}.
He then shows that under suitable conditions the spectrum of the ring $O_X(X)^G$, along with a suitable morphism of schemes, give the desired quotient structure.

All that said, however, the functions in $\mathcal O_X(X)^G$, i.e. the objects of our interest, are still classes of polynomial functions on $X$. 
For the reader who is willing to take the statements of a few theorems on faith, it is not necessary to understand anything about schemes, locally ringed spaces, or geometric invariant theory in order to appreciate the main result of this paper.
\end{remark}

\begin{lemma}\label{lem:closed}
The orbits of $G$ in $X$ are closed.
\end{lemma}
\begin{proof}
The action map $\cdot:G\times X \rightarrow X$ is a morphism of varieties.
Then by \cite[8.3]{MR0396773}, in the Zariski topology one has for each orbit $G\cdot\mathbf F$ one has that
\begin{align*}
\overline{G\cdot\mathbf F}&=G\cdot\mathbf F \bigcup_{\mathbf E\in (G\cdot\mathbf F)'} G\cdot\mathbf E  
\end{align*}
and orbits of minimal dimension are closed.
The stabilizer of $\mathbf F$ in $G$ is independent of $\mathbf F$\cite[Section 4.1]{2013arXiv1305.2229D}, so all orbits are of the same dimension\cite[25.1.3]{MR2146652}, and thus closed.  
\end{proof}

\begin{corollary}
The orbits of $G$ in $X$ are irreducible.
\end{corollary}
\begin{proof}
Follows from \cite[8.2, (d)]{MR0396773}, connectedness of $G$, and Ocneanu rigidity.
\end{proof}

\begin{proposition}\label{prop:gaugequotient}
Let $X$ be a fusion scheme and $G$ the associated gauge group acting as in Definition \ref{def:gaugetransform}.
Then there exists an algebraic scheme $Y$ and a surjective morphism of schemes $$\psi:X\rightarrow Y$$ such that 
\begin{enumerate}
	\item $\mathcal O_Y(Y)=\mathcal O_X(X)^G$,
	\item for every point $y\in Y$, $\psi^{-1}(y)$ is an orbit and $\mathcal O_{Y,y}=\mathcal O_X(\psi^{-1}(y))^G$.
\end{enumerate}
\end{proposition}
\begin{proof}
The construction in \cite[Theorem 1.1]{MR1304906} constructs $$Y=Spec(\mathcal O_X(X)^G)$$ and $\psi$ which satisfy (1) whenever $G$ is a reductive group acting algebraically on $X$. 
In this case, $G$ is a direct product of reductive groups, and thus reductive, and the action is clearly algebraic. 
In these circumstances, \cite[Amplification 1.3]{MR1304906} asserts that (2) holds if and only if orbits are closed, which has already been shown.
\end{proof}

\begin{remark}
Since, in our case, $Y$ is an algebraic scheme, up to isomorphism of schemes, it may be represented as an algebraic set. 
However, we wish to use information from the sheaf structure in Proposition \ref{prop:gaugequotient} which is not isomorphism-invariant and without knowing the orbits beforehand there is no good way of choosing an isomorphism which preserves the sheaf structure on the nose. 
\end{remark}

We can now move on to proving the first half of Theorem \ref{thm:main1}:

\begin{proposition}\label{prop:gorbs}
Let $\mathscr N\geq 2$ be the number of orbits of $G$ in $X$.
Then there exist $\mathscr N-1$ $G$-invariant regular functions $f_1,\ldots f_{\mathscr N-1}$ on $X$ such that for all $\mathbf F, \mathbf F'\in X$ $\mathbf F\sim \mathbf F'$ if and only if
\begin{displaymath}
f_i(\mathbf F)=f_i(\mathbf F')
\end{displaymath}
for $i=1,\ldots, P$.
\end{proposition}
\begin{proof}
By \cite[Corollary 1.2]{MR1304906}, for any two distinct orbits $G\cdot\mathbf F$ and $G\cdot\mathbf F'$, there exists a $G$-invariant regular function $f:X\rightarrow k$ such that $f(G\cdot\mathbf F)=\{1\}$ and $f(G\cdot\mathbf F')=\{0\}$.
Since the orbits are disjoint, $f$ is the gluing of a function on $G\cdot\mathbf F$ with a function on $G\cdot\mathbf F'$.
One then extends $f$ to a function such that $f(G\cdot\mathbf F)=\{1\}$ and $f(G\cdot\mathbf F')=\{0\}$ on all orbits $G\cdot\mathbf F'\neq G\cdot\mathbf F$. 
We only need $\mathscr N-1$ functions since for the $\mathscr N$-th orbit, $f_i(G\cdot\mathbf F_{\mathscr N})=0$ for all $f_i$ and by construction this is the only orbit for which that is the case.

Now assume that $\mathbf F_p\sim \mathbf F_q$.
The $f_i$ are defined on all of $X$ and are invariant under the $G$-action, so $f_i(\mathbf F_p)=f_i(\mathbf F_q)$ for all $f_i$.
\end{proof}

We now turn our attention to the action of $Aut(R,B)$ on $X$.
\begin{definition}\label{def:autaction}
Define the action of $\rho\in Aut(R,B)$ on $X$ via
\begin{equation}\label{eq:aut}
(\rho \cdot F)_{abc}^{d}:=F_{\rho^{-1}(a)\rho^{-1}(b)\rho^{-1}(c)}^{\rho^{-1}(d)}.
\end{equation}
Since \eq{F1},\eq{F2}, and \eq{F3} are symmetric under the action of $\rho$,  $(\rho \cdot \mathbf F)\in X$. 
Call $\mathbf F,\mathbf F'\in X$ {\em permutation equivalent} if $(\rho \cdot \mathbf F)=\mathbf F'$ for some $\rho \in Aut(R,B)$.
\end{definition}

For gauge equivalent $\mathbf F$ and $\mathbf F^g$, the triple $\mathcal F=(Id,Id,g)$ forms an object fixing monoidal functor $\mathcal F:\mathcal C(\mathbf F)\rightarrow \mathcal C(g\cdot \mathbf F)$.
$\mathcal G_\rho=(\rho,Id,Id)$ defines an object-permuting equivalence functor $\mathcal G_\rho:\mathcal C(\mathbf F) \rightarrow \mathcal C(\rho\cdot \mathbf F)$. 
For any $\mathbf F,\mathbf F^{eq}\in X$ such that there exists a monoidal equivalence $\mathcal H:\mathcal C(\mathbf F)\rightarrow \mathcal C(\mathbf F^{eq})$, $\mathcal H$ factors into the composition of a permutation and gauge equivalence.

$Aut(R,B)$ acts on $G$ by permuting labels and so we can define the action of $Aut(R,B)\ltimes G$ on $X$.
It was shown in \cite{2013arXiv1305.2229D} that for a fusion category $\mathcal C$ associated to $X$, $Aut(R,B)\ltimes G \cong Aut^\otimes(\mathcal C)$ - the group of monoidal auto-equivalences of $\mathcal C$.
$Y$ inherits an action by $Aut(R,B)$ and since $Y$ has finitely many points, there are necessarily finitely many orbits of $Aut(R,B)$ in $Y$ and they are necessarily closed.
We can now prove the second half of Theorem \ref{thm:main3}:

\begin{proposition}\label{prop:monquotient}
Let $X$ be a fusion scheme and $Y$ be as in Proposition \ref{prop:gaugequotient}.
Then there exists an algebraic scheme $Z$ and a surjective morphism of schemes $$\eta:Y\rightarrow Z$$ such that 
\begin{enumerate}
	\item $\mathcal O_Z(Z)=\mathcal O_{Y}(Y)^{Aut(R,B)}=\mathcal O_{X}(X)^{Aut(R,B)\ltimes G}$,
	\item for every point $z$ in $Z$, $\eta^{-1}(z)$ and $(\psi\circ \eta)^{-1}(z)$ are orbits in $Y$ and $X$ respectively and $$\mathcal O_{Z,z}=\mathcal O_{Y}(\eta^{-1}(z))^{Aut(R,B)}=\mathcal O_{X}((\psi\circ\eta)^{-1}(z))^{Aut(R,B)\ltimes G}.$$ 
\end{enumerate}
\end{proposition}
\begin{proof}
Similar to Proposition \ref{prop:gaugequotient} the construction $$Z=Spec(\mathcal O_Y(Y))^{Aut(R,B)\ltimes G}$$ is used.
The first statement is always true \cite[Example 6.1]{MR2004511} and the second follows by composition of $\psi \circ \eta$ in the category of $k$-schemes. 
\end{proof}

The second half of Theorem $\ref{thm:main1}$ then follows by the same logic as Proposition \ref{prop:gorbs}.
\begin{corollary}
Let $\mathbf F\in X$ and $z = (\psi \circ \eta)(\mathbf F)$
Then up to monoidal equivalence, the ring $\mathcal O_{Z,z}$ is a complete invariant of $\mathcal C(\mathbf F)$.
\end{corollary}

\subsection{Invariants for multiplicity free fusion schemes}\label{ss:ifs}
Assuming that $(R,B)$ is multiplicity free, the stronger statements in \ref{thm:main2} may be proved.
Recall that multiplicity free means $N_{ab}^c\in \{0,1\}$ for all $\gamma_{ab}^c\in \Gamma$. 
$G$ is the set of functions $g:\Gamma \rightarrow k^\times$.
Define $g_{ab}^c=g(\gamma_{ab}^c)$ and $g^{ab}_c=(g(\gamma_{ab}^c))^{-1}$.
$G$ acts on $X$ via
\begin{equation}\label{eq:gauge}
\tenj{(g\cdot F)_{abc}^d,1,e,1,1,f,1}=(g_{ab}^e)(g_{ec}^d) \tenj{F_{abc}^d,1,e,1,1,f,1}(g^{bc}_f)(g^{af}_d).
\end{equation}
which is just what one obtains from \eq{gaugemult} in the multiplicity free case.
Note that since $g_{ab}^c\neq 0$, $\tenj{(F^{g})_{abc}^d, 1,e,1,1,f,1} = 0$ if and only if $\tenj{F_{abc}^d,1,e,1,1,f,1}=0$. 
Thus, the set $F^{-1}(0)$ is a gauge invariant.

For any set $S$, define $A(S)$ to be the free Abelian group with basis $S$. 
For any $g\in G$, define $g_A:A(\Gamma)\rightarrow k^\times$ the extension of $g$ to a group homomorphism.
Rewriting \eq{gauge} reveals that $\mathbf F$ and $\mathbf F'$ are equivalent if and only if there exists a function $g_A:A(\Gamma)\rightarrow k^\times$ such that for all $\{\gamma_{ab}^c,\gamma_{ec}^d,\gamma_{bc}^f,\gamma_{af}^d\}\subset \Gamma$,
\begin{equation}\label{eq:tobegauge}
g_A(\gamma_{ab}^e \gamma_{ec}^d (\gamma_{bc}^f \gamma_{af}^d)^{-1})\tenj{F_{abc}^d,1,e,1,1,f,1}=\tenj{(F')_{abc}^d,1,e,1,1,f,1}
\end{equation}

Define $t:\Phi \rightarrow A(\Gamma)$ such that
\begin{displaymath}
t(\tenj{\Phi_{abc}^d,1,e,1,1,f,1}))=\gamma_{ab}^e\gamma_{ec}^d (\gamma_{bc}^f \gamma_{af}^d)^{-1}
\end{displaymath}
and $t_A:A(\Phi)\rightarrow A(\Gamma)$ its extension to to a group homomorphism.
Then given a gauge transformation $g\in G$, the action of $g$ on $\mathbf F\in X$ can then be rewritten as
\begin{equation}\label{eq:gaugeab}
\tenj{(F^g)_{abc}^{d},1,e,1,1,f,1}=(t\circ g_A)(\tenj{\Phi_{abc}^d,1,e,1,1,f,1})\tenj{F_{abc}^{d},1,e,1,1,f,1}
\end{equation}

\begin{lemma}\label{lem:big}
Let $\mathbf F, \mathbf F'\in X$.
Then $\mathbf F \sim \mathbf F'$ if and only if there exists a homomorphism $v:im(t_A)\rightarrow k^\times$ such that $ \phi \in \Phi$,
\begin{equation}\label{eq:phi}
(t\circ v)(\phi)F(\phi)=F'(\phi)
\end{equation}
\end{lemma}
\begin{proof}
Let $g$ be a gauge transformation from $\mathbf F$ to $\mathbf F'$.
Then $g_A$ restricts to $im(t_A)$.

Conversely, suppose we have $v:im(t_a)\rightarrow k^\times$ satisfying equation \eq{phi} for every $\phi\in\Phi$.
Since $k$ is of characteristic $0$, $k^\times$ is an injective $\Z$-module.
Thus $v$ extends to a well-defined homomorphism $\bar v:A(\Gamma)\rightarrow k^\times$, the set map restriction to $\Gamma$ of which satisfies equations \eq{tobegauge}.
\end{proof}

Define $A_{Inv}(\Phi):=ker(t_A)$ and consider an element $A_{Inv}(\Phi)\ni m = \prod_{i=1}^{n}(\phi_i)^{k_i}$.
Let $F(m)=\prod_{i=1}^{n}F(\phi_i)^{k_i}$ be the evaluation of $m$ at $\mathbf F\in X$.
If $F(m)$ is defined, then for all $g\in G$ the value $(g\cdot F)(m)$ is determined by $g$.
By definition though, $m$ is invariant under gauge transformation, and so $F(m)=(g\cdot F)(m)$.
Thus if $F(m)$ is defined at $\mathbf F'\in G\cdot\mathbf F$, then $F(m)$ is defined at all points in $G\cdot\mathbf F$.
Additionally, $F(m)$ is constant on $G\cdot\mathbf F$, and so there exists $f_m \in \mathcal O_{X}(G\cdot\mathbf F)^G$ such that $f_m=F(m)$.

\begin{proposition}\label{prop:evals}
Let $\psi:X\rightarrow Y$ be a quotient map as in Proposition \ref{prop:gaugequotient}. 
Fix $\mathbf F\in X$ with $\psi(\mathbf F)=y\in Y$.
The evaluations of monomials  
\begin{align*}
M_{y}&=\{m \in A_{Inv}(\Phi)\vert F(m)\text{ is well defined}\} \\
&=\{m \vert f_m \in \mathcal O_{X}(\psi^{-1}(y))^G\}.
\end{align*}
uniquely determine the evaluations of all $f_m\in \mathcal O_{Y,y}$.
\end{proposition}
\begin{proof}
$M_{y}$ is a finitely generated multiplicative sub-monoid of $A_{Inv}(\Phi)$. 
By definition of $\mathcal O_{Y,y}$ there exists a monoid homomorphism $$\zeta_M:M_{y}\rightarrow \mathcal O_{Y,y}$$ such that $\zeta(m)=f_m$.

Let $f_p\in \mathcal O_{Y,y}$.
Locally, it is a ratio of polynomials $g=g_1+\ldots+g_n$ and $h=h_1+\ldots+h_m$ not identically zero such that $f_p=g/h$ on $\psi^{-1}(y)$.
Without loss of generality, it may be assumed that neither $g$ nor $h$ contains variables $\phi_i$ such that $F(\phi_i)=0$.
It may also be assumed that 1) $g$ is monomial - since $g/h$ can be broken into the ($G$-invariant) terms $g_1/h+\ldots+g_n/h$ - and 2) $F(g)\neq 0$. 

This then implies that for all monomials $h_1,\ldots, h_m$ that $t_A(g)=t_A(h_j)$ for $j=1,\ldots, m$.
Thus we have that, on $\psi^{-1}(y)$, $$f_p=\frac {g} h=\frac 1 {g^{-1}(h_1+\ldots + h_m)}=\frac 1 {h_1(g^{-1})+\ldots h_n(g^{-1})}.$$
By construction the terms $h_j(g^{-1})$ are nonzero on $\psi^{-1}(y)$ and thus $h_j(g^{-1})\in M_{y}$. 
Thus $f_p$ is representable by the inverse of a sum of terms in $M_y$.
\end{proof}

Fix $y\in Y$ and define the maximal subgroup $M_y\supset A_{y}=\{m\in M_{y}\vert \forall \mathbf F\in \psi^{-1}(y), F(m)\neq 0\}$. 
Since $A_{y}$ is finitely generated, it has a basis $\mathcal B(A_{y})=\{m_1,\ldots, m_n\}$. 
In Corollary \ref{cor:last} we show that $A_y$ uniquely determines $F^{-1}(0)$ and thus $M_y$.

\begin{corollary}\label{cor:group}
Fix a point $\mathbf F\in X$ such that $\psi(\mathbf F)=y$.
The evaluations of monomials in $\mathcal B(A_y)$ uniquely determine the evaluations of all $f_m\in \mathcal O_{Y,y}$.
\end{corollary}
\begin{corollary}
Let $\mathbf F'\in X$ with $y' = \psi(\mathbf F')$, and $y\neq y'$.
Then there exists a $G$-invariant monomial $$\prod_{i=1}^n \phi_{i}^{k_i}=m \in \mathcal B(A_y)$$ such that $F(m)\neq F'(m)$ or $F'(m)$ is undefined.\label{itm:diff}
\end{corollary}
\begin{corollary}
$\mathcal C(\mathbf F)$ is gauge equivalent to its Galois conjugate if and only if $F(m)\in \Q$ for all $m\in \mathcal B(A_y)$.\label{itm:galois}
\end{corollary}

The first half of Theorem \ref{thm:main2a} has now been proved, as has the first half of Theorem \ref{thm:main2}.
Before proving the second half of Theorem \ref{thm:main2}, we note that we have all of the information necessary to characterize the quotient $Y$.

Let $y_1,\ldots y_n\in Y$ be a collection of points with $A_{y_1},\ldots, A_{y_n}$ respectively. 
Define the group
\begin{align*}
A_{y_1,\ldots, y_n} &= \{m\in A_{inv}(\phi) \vert \forall \mathbf F\in \bigcup_{i=1}^n \psi^{-1}(y_i), F(m)\neq 0\}. \\
\end{align*}
It is immediate that $A_{y_1,\ldots, y_n}=\bigcap_{i=1}^n A_{y_i}$ and so one make choices such that $$\mathcal B(A_{y_1,\ldots,y_n})=\bigcap_{i=1}^n\mathcal B(A_{y_i}).$$

\begin{theorem}\label{thm:gaugeinvs}
Partition $Y$ into set $Y_k$ such that for all $y\in Y_j,$ and $y'\in Y_k$, $A_y = A_{y'}$ if and only if $j=k$. 
The $Y_k$ define closed subschemes which define a partition on $Y$.
\end{theorem}
\begin{proof}
$Y$ is zero dimensional and all points are closed so $Y_k$ is clopen - hence it has the canonical structure of an open-subscheme.
Then for all $U\subset Y_k$, we define $\mathcal O_{Y_k}(U):=\mathcal O_{Y\vert Y_k}(U)=\mathcal O_{Y}(U)$.
There is a canonical injection of $Y_k$ into $Y$ which is trivially a closed immersion, so $Y_k$ is closed.
$Y$ is then the disjoint union of these schemes.
\end{proof}
\begin{corollary}
Each $y\in Y_k$ is uniquely determined by the evaluations of monomials in $\mathcal B(A_{Y_k})$.
\end{corollary}

In the multiplicity-free case, the action of $\rho \in Aut(R,B)$ on $\mathbf F$ as defined in \eq{aut} is equivalent to the action
\begin{equation}\label{eq:autf}
\rho \cdot \tenj{F_{abc}^d,1,e,1,1,f,1}=\tenj{F_{\rho(a)\rho(b)\rho(c)}^{\rho(d)},1,\rho(e),1,1,\rho(f),1}.
\end{equation}
For all $\mathbf F\in X$, there is an induced an action of $Aut(R,B)$ on the orbit $G\cdot \mathbf F$ via $\rho\cdot (G\cdot\mathbf F)=G\cdot(\rho \cdot \mathbf F)$ and hence an action of $Aut(R,B)$ on $Y$.
Define the action of $Aut(R,B)$ on $\Phi$ via
\begin{equation}\label{eq:autphi}
\rho \cdot \tenj{\Phi_{abc}^d,1,e,1,1,f,1}=\tenj{\Phi_{\rho(a)\rho(b)\rho(c)}^{\rho(d)},1,\rho(e),1,1,\rho(f),1}.
\end{equation}
With this action, $F$ is an $Aut(R,B)$-equivariant map. 

Fix $y\in Y$.
$M_{y}$ extends linearly to a $k$-algebra $R_{y}$ and $\zeta_M$ extends to a $k$-algebra homomorphism $\zeta_R$.
Define $R_{y_1,\ldots, y_n}$ in a similar fashion to $A_{y_1,\ldots, y_n}$.

\begin{proposition}
Let $Aut(R,B)$ act on $\Phi$ as in \eq{autphi}.
Then for all $y\in Y$ and $\rho\in Aut(R,B)$ there exists an isomorphism $\rho_R:R_{y}\rightarrow R_{\rho\cdot y}$. 
\end{proposition}
\begin{proof}
For $m=\prod_{i=1}^n(\phi_i)^{k_i}\in M_{y}$ one extends the action of \eq{autphi} to $\rho \cdot m :=\prod_{i=1}^n(\rho\cdot \phi_i)^{k_i}\in M_{\rho.y}$.
\end{proof}

\begin{proposition}\label{prop:permclosed}
$Aut(R,B)$ permutes the closed subschemes $Y_1,\ldots, Y_k$.
\end{proposition}
\begin{proof}
Let $Y'$ be a closed subscheme of $Y$ as defined in \ref{thm:gaugeinvs}.
Then $$R_{Y'}=\bigcap_{i=1}^n R_{y_i}=R_{y_1}.$$
$Aut(R,B)$ permutes points of $Y$, so $\rho \cdot y_i \neq \rho \cdot y_j$ if and only if $y_i \neq y_j$.
The set $\rho\cdot Y'=\{\rho\cdot y_1, \ldots, \rho\cdot y_n\}$ is a closed sub-scheme for exactly the same reasons as in Theorem \ref{thm:gaugeinvs}.
If $A_{\rho\cdot y_1}\neq A_{y_1}$ then $Y' \cap (\rho\cdot Y')=\emptyset$.
\end{proof}
\begin{corollary}\label{cor:permeq}
$\mathbf F,\mathbf F'\in X$ are monoidally equivalent if and only if there exists $\rho \in Aut(R,B)$ such that $A_{\rho \cdot y} = A_{y}$ and $F(m)=F'(\rho\cdot m)$, for all $m\in \mathcal B(A_{y})$.
\end{corollary}

\begin{remark}
Theorem \ref{thm:gaugeinvs} through Corollary \ref{cor:permeq} provide useful machinery for studying orbits.
Given a set $S$ of solutions to a set of pentagon equations, these may first be grouped according to the relation that for all $\mathbf F, \mathbf F'\in S$, $\mathbf F \sim \mathbf F$ if and only if there exists $\rho\in Aut(R,B)$ such that $F^{-1}(0)= \rho\cdot F'^{-1}(0)$.
This guarantees that for $Y=\sqcup_{i=1}^k Y_i$ and $\psi$ as in Proposition \ref{prop:gaugequotient}, $\mathbf F \sim \mathbf F'$ if and only if there exists $Y_i$ such that $\psi(\mathbf F) \in Aut(R,B)\cdot Y_i$ and $\psi(\mathbf F')\in Aut(R,B)\cdot Y_i$.

From Proposition \ref{prop:permclosed}, one is free to choose a representative $F^{-1}(0)$. 
One determines the sets $Y_k$ and $\eta(Y_k)$ using Proposition \ref{cor:group} and Corollary \ref{cor:permeq} where $\rho$ need only be in the group
$$Stab(Y_k)=\{\rho \in Aut(R,B)\vert \forall \mathbf F\in \psi^{-1}(Y_k), \rho \cdot F^{-1}(0)=F^{-1}(0)\}.$$ 
The number of gauge equivalence classes in the set $Aut(R,B)\cdot Y_k$ is then $\vert Y_k \vert \vert Aut(R,B) : Stab(Y_k)\vert$.
An example of this is given in Section \ref{ss:ds3} in computing the gauge and monoidal classes for categories Grothendieck equivalent to $Rep(D(S_3))$.
\end{remark}

Let $Z$ and $\eta$ be as in Proposition \ref{prop:monquotient} and choose $z\in Z$.
On $R_{\eta^{-1}(z)}$ the isomorphisms $\rho_R$ (properly restricted) are in fact automorphisms, and one recovers an action of $Aut(R,B)$ on $R_{\eta^{-1}(z)}$.

\begin{proposition}\label{thm:moninvs}
Let $\eta:Y\rightarrow Z$ be a quotient map as in Proposition \ref{prop:monquotient} and fix $z\in Z$.
The evaluations of $$R_{z}=\{r \in R_{\eta^{-1}(z)} \vert \rho \cdot r = r, \forall \rho \in Aut(R,B)\}$$
uniquely determine the evaluations of all $f_r\in \mathcal O_{Z,z}$.
\end{proposition}
\begin{proof}
The proof of this is similar to the first part of the proof of Proposition \ref{prop:evals}.
One starts with the map $\zeta_R:R_z\rightarrow \mathcal O_{Z,z}$ and notes that for any rational function $g/h$ with $g,h\in R_{\eta^{-1}(z)}$ to be $Aut(R,B)$-invariant both $g$ and $h$ must be $Aut(R,B)$-invariant.
\end{proof}
\begin{corollary}
Let $\mathbf F, \mathbf F'\in X$ such that $(\eta \circ \psi)(\mathbf F)=z$ and $(\eta \circ \psi)(\mathbf F')=z'$ and $z \neq z'$. 
Then there exists an $Aut(R,B)$-invariant linear combination of $G$-invariant monomials $l\in R_{z}$ such $F(l)\neq F'(l)$ or $F'(l)$ is undefined.
\end{corollary}
With this the proof of \ref{thm:main2} is finished.

\begin{remark}
It is important to note that the elements of $R_{z}$ are not determined by the elements of $\{m\in M_{\eta^{-1}(z)} \vert \rho \cdot m = m, \forall \rho \in Aut(R,B)\}$.
As a counter-example consider a fusion category Grothendieck equivalent to $\Z_3$ with $L=\{1,\omega,\omega^2\}$.
The involution $\rho:L\rightarrow L, \rho(1)=1, \rho(\omega)=\omega^2,\rho(\omega^2)=\omega$ is an automorphism of the fusion rules.
The products 
\begin{align*}
&\tenj{\Phi_{\omega \omega \omega}^1,1,\omega^2,1,1,\omega^2,1}\tenj{\Phi_{\omega \omega^2 \omega}^\omega,1,1,1,1,1,1}\tenj{\Phi_{\omega 1 \omega}^\omega,1,\omega,1,1,\omega,1} & and \\
&\tenj{\Phi_{\omega^2 \omega^2 \omega^2}^1,1,\omega,1,1,\omega,1}\tenj{\Phi_{\omega^2 \omega \omega^2}^{\omega^2},1,1,1,1,1,1}\tenj{\Phi_{\omega^2 1 \omega^2}^{\omega},1,\omega^2,1,1,\omega^2,1}
\end{align*}  
are not invariant under $\rho$ but their sum is.
\end{remark}

\section{Computing Invariants in Practice}\label{s:practical}

In general, it's going to be difficult to determine the sets for Theorem \ref{thm:main1} and Theorem \ref{thm:main2}.
However, for the multiplicity-free case, this is surprisingly straightforward. 

In Lemma \ref{lem:equiv} we prove that for $\mathbf F$ and $\mathbf F'\in X(R,B)$ with $\psi(\mathbf F)=y$ and $\psi(\mathbf F')=y'$, $$F^{-1}(0)=(F')^{-1}(0)\Leftrightarrow M_y = M_{y'} \Leftrightarrow A_y = A_{y'}.$$
Then for $\mathbf F\in X(R,B)$ with $\psi(\mathbf F)=y\in Y(R,B)$ and $F^{-1}(0)$, Algorithm \ref{alg:invs} computes a basis $\mathcal B(A_y)$ which from \ref{cor:group} is sufficient to check for gauge equivalence. 

\begin{algorithm}
\caption{Computing a basis for gauge invariant products}
\label{alg:invs}
\begin{algorithmic}
\Procedure{Invs}{$F^{-1}(0)$} 
	\State $\Phi^\times\gets \Phi-F^{-1}(0)$
	\For{$\phi_i \in \Phi^\times$}
		\State $\mathbf A_{ij}\gets (k_{ij}\vert t(\phi_i)=\prod_{j=1}^{\vert \Gamma \vert} \gamma_j^{k_{ij}})$ \Comment{$A_{ij}$ is the power of $\gamma_j$ in $t(\phi_i)$.}
	\EndFor
	\State $(\mathbf H,\mathbf H\mathbf A)\gets hnf(\mathbf A)$ \Comment{Compute the Hermite Normal Form of $\mathbf A$}
	\State $\mathcal H\gets \{\mathbf H_i \vert (\mathbf H \mathbf A)_i \text{ is a zero vector.}\}$.
	\ForAll{$\mathbf H_i\in \mathcal H$}
		\State $m_i \gets \prod_{j=1}^{\vert \Phi^\times \vert} \phi_j^{\mathbf H_{ij}}, \phi_j \in \Phi^\times$\Comment{Compute the monomials.}
	\EndFor
\EndProcedure
\State {\bf return} $\mathcal B(A_y)=\{m_i \vert \mathbf H_i \in \mathcal H\}$.\Comment{Return the basis $\mathcal B(A_y)$.}
\end{algorithmic}
\end{algorithm}

Again, for brevity fix $X$, $\Gamma$, and $\Phi$.
The image of $\phi\in \Phi$ under $t_{A}$ is an element in $A(\Gamma)$, which can be represented as a vector of length $\vert \Gamma \vert$ with integer coefficients.
This yields a $\vert \Gamma \vert \times \vert \Phi \vert$ matrix representing a system of $\vert \Phi \vert$ linear equations in $\vert \Gamma \vert$ variables which is equivalent to the equations \eq{gauge}.

Define $\Phi^\times=\Phi-F^{-1}(0)$ and the resulting $\vert \Gamma \vert \times \vert \Phi^\times \vert$ matrix to be $\mathbf A$.
This system of equations is canonically solved by choosing orderings of $\Gamma$ and $\Phi^\times$ then computing the Hermite normal form $\mathbf H \mathbf A $. 
Computing $\mathbf H \mathbf A$ is at worst an $\mathcal O(\vert \Phi\vert\vert \Gamma \vert^4 \log^2(\mathcal M))$\cite{MR2049753} computation where $\mathcal M$ is a bound on the entries of $\mathcal A$.

Given an ordering on $\Phi^\times$, the $\vert \Phi^\times \vert\times \vert \Phi^\times \vert$ matrix $\mathbf H$ is unique.
Rows of $\mathbf H$ correspond to products of $\phi$'s whose gauge transformation is given by the same row of $\mathbf H \mathbf A$.
Thus the product in row $i$ of $\mathbf H$ is gauge invariant if and only if row $i$ of $\mathbf H\mathbf A$ is identically zero.
Those rows of $\mathbf H$ which correspond to identically zero rows of $\mathbf H \mathbf A$ define a basis for $A_{y}$.
This proves the second half of Theorem \ref{thm:main2a}.

\begin{remark}
It is a natural question to ask whether or not the invariants from Theorem \ref{thm:main2} can be obtained without solving the pentagon equations - i.e. whether or not one can compute $Y$ directly.
In Section \ref{ss:exfib} we provide a complete example from start to finish and discuss what may be necessary in general.
\end{remark}
\section{Examples for $G$-invariants and Quotient schemes}\label{s:examples}
For simplicity, we will adopt the following notation for the rest of the section:
\begin{align*}
\Phi_{abc}^{d;ef}&:=\tenj{\Phi_{abc}^d,1,e,1,1,f,1} \\
\end{align*}

In addition to the following examples, many more are computable from the Mathematica package listed in the introduction.

\subsection{Quotient Scheme for Fibonacci categories}\label{ss:exfib}

As a first example which is short enough to compute explicitly, we re-derive the classification of fusion categories Grothendieck equivalent to representation category of the even half of $\mathfrak{su}(2)_3$.
These are the Fibonacci and Yang-Lee categories.
We will denote their Grothendieck ring by $\mathbf{Fib}$ with basis elements $\{1,\tau\}$ such with the non-trivial fusion rule $\tau \otimes \tau \cong 1 \oplus \tau$.
The elements of $\Phi(\mathbf{Fib})$ and images under $t_A$ are given in Table \ref{tab:fibvars}

The first step is to identify the sets $F^{-1}(0), \forall \mathbf F \in X$.
It is well known that for all $\mathbf F\in X(\mathbf{Fib})$, $F_{abc}^{d;ef}\neq 0$ for all $\Phi_{abc}^{d;ef}\in \Phi(\mathbf{Fib})$. 
This can be determined by categorical considerations and by the procedure in \cite[2.5.2]{Bonderson2007}.
There are no automorphisms, and thus $Z(\mathbf{Fib})=Y(\mathbf{Fib})$.
Then, using the ordering in \ref{tab:fibvars}, the matrix $\mathbf A$ coming from equations of the form \eq{gauge}, $\mathbf H$, and $\mathbf H \mathbf A$ are given in Tables \ref{tab:a} and \ref{tab:ha}.
After reordering, these yield the basis for $A_{\mathbf{Fib}}$ given in Table \ref{tab:afibbasis}.

\begin{table}[t]
\makebox[\textwidth][c]{
\begin{tabular}{| c | c || c | c|| c | c || c | c |}
\hline
$\Phi_{111}^{1;11}$ & 1 & $\Phi_{1 1 \tau}^{\tau; 1 \tau}$ & $t_{1 1}^1 (t_{1\tau}^\tau)^{-1}$ & $\Phi_{1\tau 1}^{\tau; \tau \tau}$ & $ 1 $ &  $\Phi_{1 \tau \tau}^{1;\tau 1}$ & $t_{1\tau}^\tau(t_{1 1}^1)^{-1}$ \\ \hline
$\Phi_{1\tau \tau}^{\tau; \tau \tau}$ & 1 & $\Phi_{\tau 1 1}^{\tau; \tau 1}$ & $t_{\tau 1}^\tau (t_{11}^1)^{-1}$ & $\Phi_{\tau 1 \tau}^{1;\tau \tau}$ & $t_{\tau 1}^\tau (t_{1\tau}^\tau)^{-1}$ & $\Phi_{\tau 1 \tau}^{\tau; \tau \tau}$ & $t_{\tau 1}^\tau (t_{1 \tau}^\tau)^{-1}$ \\ \hline
$\Phi_{\tau \tau 1}^{1;1 \tau}$ & $t_{1 1}^1 (t_{\tau 1}^\tau)^{-1}$ & $\Phi_{\tau \tau 1}^{\tau; \tau \tau}$ & 1 & $\Phi_{\tau \tau \tau}^{1; \tau \tau}$ & 1 & $\Phi_{\tau \tau \tau}^{\tau; 1 1}$ & $t_{1 \tau}^\tau (t_{\tau 1}^\tau)^{-1}$ \\ \hline
$\Phi_{\tau \tau \tau}^{\tau; 1 \tau}$ & $t_{\tau \tau}^1 t_{1\tau}^\tau(t_{\tau \tau}^\tau)^{-2}$ & $\Phi_{\tau \tau \tau}^{\tau; \tau 1}$ & $(t_{\tau \tau}^\tau)^2 (t_{\tau \tau}^1 t_{\tau 1}^\tau)^{-1}$ & $\Phi_{\tau \tau \tau}^{\tau; \tau \tau}$ & 1 & & \\ \hline
\end{tabular}}
\caption{The variables $\Phi$ for Fibonacci and Yang-Lee categories.}
\label{tab:fibvars}
\end{table}

\begin{table}[b]
\begin{displaymath}
\mathbf A=
\left(\begin{array}{ccccc}
0 & 0 & 0 & 0 & 0\\
1 & -1 & 0 & 0 & 0 \\
0 & 0 & 0 & 0 & 0\\
-1 & 1 & 0 & 0 & 0 \\
0 & 0 & 0 & 0 & 0 \\
-1 & 0 & 1 & 0 & 0 \\
0 & -1 & 1 & 0 & 0 \\
0 & -1 & 1 & 0 & 0 \\
1 & 0 & -1 & 0 & 0 \\
0 & 0 & 0 & 0 & 0 \\
0 & 0 & 0 & 0 & 0 \\
0 & 1 & -1 & 0 & 0 \\
0 & 0 & -1 & -1 & 2 \\
0 & 1 & 0 & 1 & -2 \\
0 & 0 & 0 & 0 & 0 \\
\end{array}\right),
\mathbf{HA} = 
\left(\begin{array}{ccccc}
1 & 0 & 0 & 1 & -2 \\
0 & 1 & 0 & 1 & -2 \\
0 & 0 & 1 & 1 & -2 \\
0 & 0 & 0 & 0 & 0 \\
0 & 0 & 0 & 0 & 0 \\
0 & 0 & 0 & 0 & 0 \\
0 & 0 & 0 & 0 & 0 \\
0 & 0 & 0 & 0 & 0 \\
0 & 0 & 0 & 0 & 0 \\
0 & 0 & 0 & 0 & 0 \\
0 & 0 & 0 & 0 & 0 \\
0 & 0 & 0 & 0 & 0 \\
0 & 0 & 0 & 0 & 0 \\
0 & 0 & 0 & 0 & 0 \\
0 & 0 & 0 & 0 & 0 \\
\end{array}\right)
\end{displaymath}
\caption{The matrices $\mathbf A$ and $\mathbf{HA}$ for $X(\mathbf{Fib})$.}
\label{tab:a}
\end{table}

\begin{table}[t]
\begin{displaymath}
\mathbf{H}=
\left(\begin{array}{ccccccccccccccc}
 0 & 0 & 0 & 0 & 0 & -1 & 0 & 0 & 0 & 0 & 0 & 0 & -1 & 0 & 0 \\
 0 & -1 & 0 & 0 & 0 & -1 & 0 & 0 & 0 & 0 & 0 & 0 & -1 & 0 & 0 \\
 0 & 0 & 0 & 0 & 0 & 0 & 0 & 0 & 0 & 0 & 0 & 0 & -1 & 0 & 0 \\
 0 & 1 & 0 & 1 & 0 & 0 & 0 & 0 & 0 & 0 & 0 & 0 & 0 & 0 & 0 \\
 0 & 0 & 0 & 0 & 1 & 0 & 0 & 0 & 0 & 0 & 0 & 0 & 0 & 0 & 0 \\
 1 & 0 & 0 & 0 & 0 & 0 & 0 & 0 & 0 & 0 & 0 & 0 & 0 & 0 & 0 \\
 0 & -1 & 0 & 0 & 0 & -1 & 1 & 0 & 0 & 0 & 0 & 0 & 0 & 0 & 0 \\
 0 & -1 & 0 & 0 & 0 & -1 & 0 & 1 & 0 & 0 & 0 & 0 & 0 & 0 & 0 \\
 0 & 0 & 0 & 0 & 0 & 1 & 0 & 0 & 1 & 0 & 0 & 0 & 0 & 0 & 0 \\
 0 & 0 & 0 & 0 & 0 & 0 & 0 & 0 & 0 & 1 & 0 & 0 & 0 & 0 & 0 \\
 0 & 0 & 0 & 0 & 0 & 0 & 0 & 0 & 0 & 0 & 1 & 0 & 0 & 0 & 0 \\
 0 & 1 & 0 & 0 & 0 & 1 & 0 & 0 & 0 & 0 & 0 & 1 & 0 & 0 & 0 \\
 0 & 0 & 1 & 0 & 0 & 0 & 0 & 0 & 0 & 0 & 0 & 0 & 0 & 0 & 0 \\
 0 & 1 & 0 & 0 & 0 & 1 & 0 & 0 & 0 & 0 & 0 & 0 & 1 & 1 & 0 \\
 0 & 0 & 0 & 0 & 0 & 0 & 0 & 0 & 0 & 0 & 0 & 0 & 0 & 0 & 1 \\
\end{array}\right)
\end{displaymath} 
\caption{The matrix $\mathbf H$ for $X(\mathbf{Fib})$.}
\label{tab:ha}
\end{table}

\begin{table}[b]
\makebox[\textwidth][c]{
\begin{tabular}{| c | c || c | c || c | c |}
\hline
$s_1$ & $\Phi_{1 1 1}^{1; 1 1}$ & $s_2$ & $\Phi_{1\tau 1}^{\tau; \tau \tau}$ & $s_3$ & $\Phi_{1\tau\tau}^{\tau;\tau\tau}$ \\\hline
$s_4$ & $\Phi_{\tau\tau 1}^{\tau;\tau\tau}$ & $s_5$ & $\Phi_{\tau\tau\tau}^{1;\tau\tau}$ & $s_6$ & $\Phi_{\tau \tau \tau}^{\tau; \tau \tau}$ \\ \hline 
$s_7$ & $\Phi_{1 1 \tau}^{\tau; 1 \tau} \Phi_{1\tau\tau}^{1;\tau 1}$ & $s_8$ & $\Phi_{\tau 1 1}^{\tau; \tau 1} \Phi_{\tau \tau 1}^{1;1 \tau}$ & $s_9$ & $\Phi_{1 1 \tau}^{\tau;1\tau}\Phi_{\tau 1 1}^{\tau; \tau 1}\Phi_{\tau\tau\tau}^{\tau; 1 1}$ \\ \hline
$s_{10}$ & $\Phi_{\tau 1 \tau}^{1; \tau \tau}(\Phi_{1 1 \tau}^{\tau; 1 \tau}\Phi_{\tau 1 1}^{\tau; \tau 1})^{-1}$ & $s_{11}$ & $\Phi_{\tau 1 \tau}^{\tau;\tau\tau}(\Phi_{1 1 \tau}^{\tau; 1 \tau}\Phi_{\tau 1 1}^{\tau; \tau 1})^{-1}$ & $s_{12}$ & $\Phi_{1 1 \tau}^{\tau; 1 \tau}\Phi_{\tau 1 1}^{\tau; \tau 1}\Phi_{\tau\tau\tau}^{\tau;1\tau}\Phi_{\tau\tau\tau}^{\tau;\tau 1}$ \\\hline
\end{tabular}}
\caption{A basis for $A_{fib}$.}
\label{tab:afibbasis}
\end{table}

Next, one divides out by non-zero terms in the pentagon equations to obtain a set of gauge invariant equations in the gauge invariant monomials as follows: 
For simplicity, let $p_1$ and $p_2$ be the only two nonzero terms in some \eq{F3} with $p_1 - p_2 = 0$.
Without loss of generality, one may take $p_1 (p_2)^{-1}=m_1 \in \mathcal B(A_y)$ and obtains
\begin{align*}
(p_2)^{-1}(p_1 - p_2) &=0 \\
\Rightarrow p_1 (p_2)^{-1} - 1 &= 0 \\
\Rightarrow m_1 -1 &= 0.
\end{align*}
The last line is manifestly gauge invariant.

All equations of the form \eq{F3} can be localized in this way with respect to all non-zero terms. 
All terms of $Det(\Phi_{abc}^d)$ have the same gauge class, and so can similarly be localized.
There are only ever finitely many equations and finitely many terms and so there are only finitely many such localizations.
This yields a finite set of equations in gauge invariant monomials, i.e. an ideal $I$ in some $k[\mathcal B(A_{Y_i})]$.

Applying this procedure to $X(\mathbf{Fib})$ one finds that the evaluations for $s_1, \ldots, s_{12}$ are determined by the equations

\begin{align}
s_1=s_2=s_3=s_4=s_5=s_7=s_8=s_{10}=s_{11}=1 \label{eq:fib1}\\
s_9=s_{12}=-s_6 \label{eq:fib2}\\
(s_6)^2 - s_6 - 1 = 0\label{eq:fib3}.
\end{align}

the solutions to which are determined by $s_6=\frac 1 2 (1\pm \sqrt{5})$.
We note that the equations \eq{fib1}-\eq{fib3} are only those which determine the evaluations for all $s_i\in \mathcal B(A_{\mathbf{Fib}})$.
We have omitted most of the relations defining the associated ideal $I$, but it is straightforward to compute that $I$ is not radical.
Thus $Y$ is not itself an algebraic set.

While simple, this example encapsulates everything about the construction in Section \ref{s:quotients}.
It's easy to see that one can completely reconstruct the associativity information for the categories from the evaluations of $s_1,\ldots, s_{12}$.
In this way the result has another interpretation: $Y(\mathbf{Fib})$ is the moduli space for $\mathbf{Fib}$. 
Its points are monoidal equivalence classes of categories and these are parameterized by the zero-loci of $x^2-x-1$.

This example is made easier by the fact that for all $\mathbf F \in X(\mathbf{Fib}), F^{-1}(0)=\emptyset$. 
For categories without zeroes - e.g. the Tambara-Yamagami categories associated to a commutative group $G$ - the above could be used to compute $Y$ without modification.

In general, one must first identify the sets $F_i^{-1}(0)$ in order to compute a basis for the $A_{Y_i}$. 
For the case where $\mathcal K_0(\mathcal C)$ has unitary categories, \cite[2.5.2]{Bonderson2007} presents a heuristic algorithm which is conjectured to always determine the sets $F_i^{-1}(0)$ and thus the groups $A_{Y_i}$.
Considering the $F^{-1}(0)$ individually, one then exploits the partitioning Theorem \ref{thm:gaugeinvs}.

\subsection{Gauge and monoidal classes for $Rep(D(S_3))$}\label{ss:ds3}

In this example we use the invariants of \ref{ss:ifs} to determine gauge and monoidal equivalence classes for fusion categories Grothendieck equivalent to the representation category of the quantum double of $S_3$.
In this case, we take as our starting point a collection of solutions to the pentagon equations for $Rep(D(S_3))$.
The number of matrices involved is large so rather than printing them they have been made available at \href{http://mktitsworth.com/fusion-categories/}{http://mktitsworth.com/fusion-categories}.
This package includes the $F$ and $R$ matrices, pivotal coefficients, and a Mathematica notebook with the computations referenced here-in.
Our classification is complete, however the solutions were found using computerized methods.

The Grothendieck ring will be referred to as $Rep(D(S_3))$.
Let the basis elements be $\{1,\epsilon,\beta_1, \alpha_+,\alpha_-, \beta_2,\beta_3,\beta_4 \}$ with Frobenius-Perron dimensions $\{1,1,2,3,3,2,2,2\}$ respectively.
The fusion rules are commutative and each $\beta_i$ forms a $Rep(S_3)$ subcategory together with $\{1,\epsilon\}$.
Together the $\{1,\epsilon,\beta_i\}$ yield the adjoint subcategory $Rep(D(S_3))_{ad}$ with $\beta_i \otimes \beta_j =\beta_k \oplus \beta_l$ with $i\neq j\neq k \neq l$.
The remaining fusion rules are then given by

\begin{align*}
\epsilon \otimes \alpha_\pm &= \alpha_\mp, & \beta_i \otimes \alpha_\pm &= \alpha_+ \oplus \alpha_-, \\
\alpha_\pm\otimes \alpha_\pm &= 1 \bigoplus_{i=1}^4 \beta_i, & \alpha_\pm \otimes \alpha_\mp &= \epsilon \bigoplus_{i=1}^4 \beta_i
\end{align*}

The automorphism group for $Rep(D(S_3))$ is $\Z_2\times S_4$. 

There are five monoidal equivalence classes - two for each set of $i,j$ conditions on zeroes in the first row of Table \ref{tab:zeroes1} and one for the zeroes in the second row.
All monoidal classes admit unitary structures - i.e. all $F$-matrices can be made unitary - but only four admit braidings.
The allowable braidings correspond to the double and twisted double of $S_3$ as well as the two non-group theoretical categories discovered by \cite{MR2587410}.
The modular data for these categories is readily available there.

To our knowledge the final equivalence class is new.
In addition to those properties below which distinguish it from all others, its pivotal structures are also distinct:
In the braided cases, the Frobenius-Schur indicators for the $\alpha_{\pm}$ objects are always equal, and in the non-braided case they are opposite.
It is straightforward to compute this given the $F$-matrices.

\subsubsection{Classification by Invariants}

We proceed via our observations \ref{thm:gaugeinvs} and \ref{prop:permclosed}.
We first partition the points of $Y(Rep(D(S_3)))$ via the equivalence relation $$y_i\sim y_j \Leftrightarrow M_y=M_{y'}\Leftrightarrow F^{-1}(0)=F^{'-1}(0), \forall \mathbf F\in \psi^{-1}(y) \text{ and }\forall F'\in \psi^{-1}(y).$$
We then partition the points of $Z(Rep(D(S_3)))$ via $$z_i\sim z_j \Leftrightarrow \forall y_k\in \eta^{-1}(z_i) \exists y_l \in \eta^{-1}(z_j)\text{ and }\rho \in Aut(D(S_3))\vert M_{y_k}=M_{\rho\cdot y_l}$$ which is to say that the closed sub-schemes as defined in \ref{thm:gaugeinvs} for $y_k$ is permuted by $\rho\in Aut(R,B)$ to that of $y_l$. 
We then determine a set of nonzero invariants $m_1,\ldots,m_8$ defined on all of $Y(Rep(D(S_3)))$, each of which distinguishes between gauge equivalence classes and whose sum determines monoidal classes.

\begin{table}[t]
\makebox[\textwidth][c]{
\begin{tabular}{|l|r|}\hline
&  \\
$\Phi_{\beta_i \alpha_\pm \beta_j}^{\alpha_\pm; \alpha_\pm \alpha_\mp}$, $\Phi_{\beta_i \alpha_\pm \beta_j}^{\alpha_\pm; \alpha_\mp \alpha_\pm}$, $\Phi_{\beta_i \alpha_\pm \beta_j}^{\alpha_\mp; \alpha_\pm \alpha_\pm}$, $\Phi_{\beta_i \alpha_\pm \beta_j}^{\alpha_\mp; \alpha_\mp \alpha_\mp}$ & $i\neq j\in \{\beta_1, \beta_2\}\text{ or }i=j\in\{\beta_3,\beta_4\}$\\
& \\
$\Phi_{\alpha_\pm \beta_i \alpha_\pm}^{\beta_j; \alpha_\pm \alpha_\mp}$, $\Phi_{\alpha_\pm \beta_i \alpha_\pm}^{\beta_i; \alpha_\pm \alpha_\mp}$, $\Phi_{\alpha_\pm \beta_i \alpha_\mp}^{\beta_j; \alpha_\pm \alpha_\pm}$, $\Phi_{\alpha_\pm \beta_i \alpha_\mp}^{\beta_j; \alpha_\mp \alpha_\mp}$ & or \\
& \\
$\Phi_{\alpha_\pm \alpha_\pm \alpha_\pm}^{\alpha_\mp; \beta_i \beta_j}$, $\Phi_{\alpha_\pm \alpha_\pm \alpha_\mp}^{\alpha_\pm; \beta_i \beta_j}$, $\Phi_{\alpha_\pm \alpha_\mp \alpha_\pm}^{\alpha_\pm; \beta_i \beta_j}$, $\Phi_{\alpha_\mp \alpha_\pm \alpha_\pm}^{\alpha_\pm; \beta_i \beta_j}$ & $i\neq j\in \{\beta_1, \beta_2\}\text{ or }i\neq j\in\{\beta_3,\beta_4\}$\\
& \\ \hline
& \\
$\Phi_{\beta_i \alpha_\pm \beta_i}^{\alpha_\pm; \alpha_\pm \alpha_\mp}$, $\Phi_{\beta_i \alpha_\pm \beta_i}^{\alpha_\pm; \alpha_\mp \alpha_\pm}$, $\Phi_{\beta_i \alpha_\pm \beta_i}^{\alpha_\mp; \alpha_\pm \alpha_\pm}$, $\Phi_{\beta_i \alpha_\pm \beta_i}^{\alpha_\mp; \alpha_\mp \alpha_\mp}$ & \\
& \\
$\Phi_{\alpha_\pm \beta_i \alpha_\pm}^{\beta_i; \alpha_\pm \alpha_\mp}$, $\Phi_{\alpha_\pm \beta_i \alpha_\pm}^{\beta_i; \alpha_\mp \alpha_\pm}$, $\Phi_{\alpha_\pm \beta_i \alpha_\mp}^{\beta_i; \alpha_\pm \alpha_\pm}$, $\Phi_{\alpha_\pm \beta_i \alpha_\mp}^{\beta_i; \alpha_\mp \alpha_\mp}$ & \\
& \\
$\Phi_{\alpha_\pm \alpha_\pm \alpha_\pm}^{\alpha_\mp; \beta_i \beta_i}$, $\Phi_{\alpha_\pm \alpha_\pm \alpha_\mp}^{\alpha_\pm; \beta_i \beta_i}$, $\Phi_{\alpha_\pm \alpha_\mp \alpha_\pm}^{\alpha_\pm; \beta_i \beta_i}$, $\Phi_{\alpha_\mp \alpha_\pm \alpha_\pm}^{\alpha_\pm; \beta_i \beta_i}$ & \\
& \\
\hline
\end{tabular}
}
\caption{Table of zero sets for solutions to $Rep(D(S_3))$ pentagon equations.}
\label{tab:zeroes1}
\end{table}

Given that the solutions to the pentagon equations for $Rep(D(S_3))$ are taken as input, the zero sets are readily available.
The allowable zeroes are also constructable by using \cite[2.5.2]{Bonderson2007} iteratively - determining the zeroes for the $Rep(S_3)$ subcategories, for $Rep(D(S_3))_{ad}$, and finally $Rep(D(S_3))$.

For the $Rep(S_3)$ subcategories the only zero is $\Phi_{\beta_i,\beta_i,\beta_i}^{\beta_i;\beta_i\beta_i}$.
For $Rep(D(S_3))_{ad}$, all entries of the form $\Phi_{\beta_i \beta_j \beta_i}^{\beta_j; \beta_k \beta_k}$ are zero in all solutions and this set is invariant under permutation.
For $Rep(D(S_3))$, up to permutation there are three sets of zeroes all sharing those of $Rep(D(S_3))_{ad}$ with additional zeroes listed in Table \ref{tab:zeroes1}. 
Two sets of zeroes share form given in the first row/column of Table \ref{tab:zeroes1}.
They are distinguished by the conditions on the two dimensional objects which appear given in the first row/second column.

By inspection of Table \ref{tab:zeroes1} the first zero set is fixed by the Klein 4-group generated by $\langle ( \beta_1,\beta_2),(\beta_3,\beta_4) \rangle$ plus $(\alpha_+,\alpha_-)$.
This is a subgroup of order eight and so there are six orbits - i.e. we can obtain six sets of zeroes from those in Table \ref{tab:zeroes1}.
The second set is fixed by the dihedral group $D_4$ generated by $\langle (\beta_1,\beta_2),(\beta_1,\beta_3,\beta_2,\beta_4)\rangle$ plus $(\alpha_+,\alpha_-)$.
This is a subgroup of order sixteen so there are three orbits.
The final set of zeros in Table \ref{tab:zeroes1} is invariant under the action of $Aut(D(S_3))$.
Thus we have ten total zero sets and ten subschemes $Y_1,\ldots, Y_{10}$.

\begin{table}[b]
\makebox[\textwidth][c]{
\begin{tabular}{|l|c|c|}
\hline
Solution set, permutation 								& $Inv(D(S_3))=\{\Phi_{\alpha_\pm \alpha_\pm \alpha_\pm}^{\alpha_\pm; \beta_i \beta_i}\Phi_{\alpha_\pm \beta_i \alpha_\pm}^{\beta_i; \alpha_\pm \alpha_\pm}\vert i=1,2,3,4\}$ 				& $\underset{i\in Inv(D(S_3))}{\sum} i$ 	\\[1.5ex] \hline
Zero set 1, unpermuted 									& $\{\pm \frac 2 3, \pm \frac 2 3, \pm \frac 1 6, \pm \frac 1 6, \pm \frac 2 3, \pm \frac 2 3, \pm \frac 1 6, \pm \frac 1 6\}$ 	& $\pm \frac{10}{3}$ 						\\\hline
Zero set 1, $(\beta_2, \beta_3)$ 						& $\{\pm \frac 2 3, \pm \frac 1 6, \pm \frac 2 3, \pm \frac 1 6, \pm \frac 2 3, \pm \frac 1 6, \pm \frac 2 3, \pm \frac 1 6\}$ 	& $\pm \frac{10}{3}$ 						\\\hline
Zero set 1, $(\beta_2, \beta_3, \beta_4)$ 				& $\{\pm \frac 2 3, \pm \frac 1 6, \pm \frac 1 6, \pm \frac 2 3, \pm \frac 2 3, \pm \frac 1 6, \pm \frac 1 6, \pm \frac 2 3\}$ 	& $\pm \frac{10}{3}$ 						\\\hline
Zero set 1, $(\beta_1, \beta_3, \beta_2)$ 				& $\{\pm \frac 1 6, \pm \frac 2 3, \pm \frac 2 3, \pm \frac 1 6, \pm \frac 1 6, \pm \frac 2 3, \pm \frac 2 3, \pm \frac 1 6\}$ 	& $\pm \frac{10}{3}$ 						\\\hline
Zero set 1, $(\beta_1, \beta_3, \beta_4, \beta_2)$ 		& $\{\pm \frac 1 6, \pm \frac 2 3, \pm \frac 1 6, \pm \frac 2 3, \pm \frac 1 6, \pm \frac 2 3, \pm \frac 1 6, \pm \frac 2 3\}$	& $\pm \frac{10}{3}$ 						\\\hline
Zero set 1, $(\beta_1, \beta_3), (\beta_2, \beta_4)$ 	& $\{\pm \frac 1 6, \pm \frac 1 6, \pm \frac 2 3, \pm \frac 2 3, \pm \frac 1 6, \pm \frac 1 6, \pm \frac 2 3, \pm \frac 2 3\}$	& $\pm \frac{10}{3}$ 						\\\hline
Zero set 2, all 										& $\{\pm \frac 1 6, \pm \frac 1 6, \pm \frac 1 6, \pm \frac 1 6, \pm \frac 1 6, \pm \frac 1 6, \pm \frac 1 6, \pm \frac 1 6\}$ 	& $\pm \frac 4 3 $ 							\\\hline
Zero set 3, unpermuted 									& $\{\frac 2 3, \frac 2 3, \frac 2 3, \frac 2 3, -\frac 2 3, -\frac 2 3, -\frac 2 3, -\frac 2 3\}$ 								& 0 \\\hline
Zero set 3, $(\alpha_+,\alpha_-)$ 						& $\{-\frac 2 3, -\frac 2 3, -\frac 2 3, -\frac 2 3, \frac 2 3, \frac 2 3, \frac 2 3, \frac 2 3\}$ 								& 0 \\\hline
\end{tabular}
}
\caption{Nonzero invariants which determine the gauge and monoidal equivalence classes.}
\label{tab:nonzero}
\end{table}

At this point the classification is easy.
The invariant products $$\Phi_{\alpha_\pm \alpha_\pm \alpha_\pm}^{\alpha_\pm; \beta_i \beta_i}\Phi_{\alpha_\pm \beta_i \alpha_\pm}^{\beta_i; \alpha_\pm \alpha_\pm}$$ are defined on all of $Y$ and each takes two values on each $Y_i$.
Their values are given in Table \ref{tab:nonzero}.
Thus there are 20 gauge classes.
Their sum
$$
\sum_{i=1}^4\Phi_{\alpha_+ \alpha_+ \alpha_+}^{\alpha_+; \beta_i \beta_i}\Phi_{\alpha_+ \beta_i \alpha_+}^{\beta_i; \alpha_+ \alpha_+} +
\sum_{i=1}^4\Phi_{\alpha_- \alpha_- \alpha_-}^{\alpha_-; \beta_i \beta_i}\Phi_{\alpha_- \beta_i \alpha_-}^{\beta_i; \alpha_- \alpha_-}
$$
is invariant under the action of $Aut(D(S_3))$ and takes five values, so there are five monoidal classes.

\begin{appendices}
\section{Nonzero Invariant Monomials Determine All Invariant Monomials}\label{app:important}

For the entirety of this section let $X$ be a fusion scheme with $\mathbf F, \mathbf F' \in X$ and $Y$ a quotient scheme as in Proposition \ref{prop:evals} with $\psi(\mathbf F)=y$ and $\psi(\mathbf F')=y'$.
The point of this appendix is to prove the following:

\begin{lemma}\label{lem:equiv}
The following are equivalent:
\begin{enumerate}
	\item $F^{-1}(0)=(F')^{-1}(0)$, \label{en:1}
	\item $M_y = M_{y'}$, and \label{en:2}
	\item $A_y = A_{y'}$. \label{en:3}
\end{enumerate}
\end{lemma}
\begin{proof}
The direction \ref{en:1} $\Rightarrow$ \ref{en:2} $\Rightarrow$ \ref{en:3} is by construction.
\ref{en:3} $\Rightarrow$ \ref{en:1} is a corollary of Lemma \ref{lem:important} and Corollary \ref{cor:last}.
\end{proof}

The proof of Lemma \ref{lem:important} is made easier by using the graphical calculus for fusion categories, so for completeness, we provide a quick review.
Since our proof is for multiplicity free categories, we will omit indices at vertices.
For a full treatment of the graphical calculus see \cite{2013arXiv1305.2229D}.

Let $\mathcal C$ be a fusion category.
Basis vectors for hom spaces $\mathcal C(a\otimes b, c)$ may be represented graphically as
\begin{align*}
\trifus{a,b,c} & \in \mathcal C(a\otimes b, c) & \trispl{a,b,c} & \in \mathcal C(c,a\otimes b)
\end{align*}
Additionally we will always choose the basis for $\mathcal C(c,a \otimes b)$ which is algebraically dual to the basis for $\mathcal C(a\otimes b, c)$.
Graphically, the numbers $\tenj{F_{abc}^{d},1,e,1,1,f,1}$ are coefficients in the expression
$$\fustreel{a,b,c,d,e}=\tenj{F_{abc}^{d},1,e,1,1,f,1}\fustreer{a,b,c,d,e}$$
which relates the two different ways in which bases for $\mathcal C(a\otimes b\otimes c,d)$ spaces can be decomposed given bases for the spaces $\mathcal C(a\otimes b,c)$.
One writes a similar expression for the splitting spaces
$$\spltreel{a,b,c,d,e}=\tenj{G_{abc}^{d},1,f,1,1,e,1}\spltreer{a,b,c,d,f}$$
where $G_{abc}^d$ is the inverse matrix to $F_{abc}^d$ satisfying equivalent relations to \eq{F2} and \eq{F3}.

\begin{lemma}\label{lem:important}
Every $\phi \in \Phi$ such that $F(\phi)\neq 0$ has non-zero power in some $m \in A_{y}$.
\end{lemma}
\begin{proof}
The case where $m=\phi$ (i.e. $\phi$ is gauge invariant) is trivial. 
Recall that $F(\phi)=\tenj{F_{abc}^{d},1,e,1,1,f,1}$ for some $\{a,b,c,d,e,f\}$ and consider the morphism given by
\begin{align*}
\hspace{-20pt}\dtod{a,b,c,d,e,f} =\sum_{f'}\tenj{F_{abc}^{d},1,e,1,1,f',1}\dtodt{a,b,c,d,f,f'} = \tenj{F_{abc}^{d},1,e,1,1,f,1}\dtodt{a,b,c,d,f,f} = \tenj{F_{abc}^{d},1,e,1,1,f,1}\dtodr{d}.
\end{align*}

We then fix an object $i$ and tensor this with some $d$ to obtain

\begin{centering}
\begin{align*}
\hspace{-40pt}\tenj{F_{abc}^{i},1,f,1,1,g,1}\sum_e \propicmt{i,d,e} &= &\tenj{F_{abc}^{i},1,f,1,1,g,1}\propicmo{i,d} &=  &\propicz{a,b,c,d,i,f,g} &= &\underset{e,e'}{\sum}\propicpo{a,b,c,d,i,e,f,g} &=
\end{align*}
\begin{align*}
\hspace{-20pt}\sum_{e,g'}\tenj{F_{abc}^{i},1,f,1,1,g',1}\propicpoa{a,b,c,d,i,e,g',g} &= &\underset{e,h',j'}{\sum}\tenj{F_{abc}^{i},1,f,1,1,g,1}\tenj{F_{agd}^{e},1,i,1,1,j,1}\tenj{F_{bcd}^{j},1,g,1,1,h,1}\propicpt{a,b,c,d,i,e,g,h,j} &=
\end{align*}
\end{centering}
\begin{displaymath}
= \underset{e,h,h',j,j'}{\sum}\tenj{F_{abc}^{i},1,f,1,1,g,1}\tenj{F_{agd}^{e},1,i,1,1,j,1}\tenj{F_{bcd}^{j},1,g,1,1,h,1}\tenj{G_{agd}^{e},1,j,1,1,i,1}\tenj{G_{bcd}^{j},1,h,1,1,g,1}\propicpr{a,b,c,d,i,e,h,j} 
\end{displaymath}
\begin{equation}
=\tenj{F_{abc}^{i},1,f,1,1,g,1}\left(\underset{e,h,j}{\sum}\tenj{F_{agd}^{e},1,i,1,1,j,1}\tenj{F_{bcd}^{j},1,g,1,1,h,1}\tenj{G_{agd}^{e},1,j,1,1,i,1}\tenj{G_{bcd}^{j},1,h,1,1,g,1}\right)\propicmt{i,d,e}
\end{equation}

This morphism is a multiple $\tenj{F_{abc}^{i},1,f,1,1,g,1}$ of the identity on $i\otimes d$, and we have that at least one of the terms
\begin{displaymath}
\tenj{F_{abc}^{i},1,f,1,1,g,1}\tenj{F_{agd}^{e},1,i,1,1,j,1}\tenj{F_{bcd}^{j},1,g,1,1,h,1}
\end{displaymath}
must be non-zero. 
Each of these is an evaluation of a cubic term in some \eq{F3} and every equation of the form \eq{F3} has either zero or at least two nonzero terms.
In the multiplicity free case, given an equation of the form \eq{F3} all terms in the equation transform as $t_{ab}^f t_{f c}^i t_{i d}^e (t_{a j}^e t_{b h}^j t_{cd}^h)^{-1}$. 
The ratio of any two terms is in $A_{Inv}(\Phi)$ and the ratio of any two non-zero terms is in $A_y$.
\end{proof}

\begin{corollary}\label{cor:last}
$A_y$ uniquely determines $F^{-1}(0)$.
\end{corollary}

Lemma \ref{lem:equiv} is then a corollary of Lemma \ref{lem:important}.

\end{appendices}

\bibliographystyle{amsalpha}
{\footnotesize
\bibliography{library}
}

\end{document}